\documentclass[letterpaper]{amsart}

\usepackage[english]{babel}
\usepackage[utf8]{inputenc}
\usepackage{amsmath, amssymb, amsthm}
\usepackage{graphicx}
\usepackage[colorinlistoftodos,textsize=footnotesize]{todonotes}
\usepackage{tikz}
\usetikzlibrary{decorations.pathreplacing}
\usepackage{fullpage}

\usepackage{hyperref}
\hypersetup{
    colorlinks,
    linkcolor={red!50!black},
    citecolor={blue!50!black},
    urlcolor={blue!80!black}
}

\usepackage[capitalize, noabbrev]{cleveref}

\newtheorem{theorem}{Theorem}
\newtheorem{lemma}[theorem]{Lemma}
\newtheorem{corollary}[theorem]{Corollary}

\newtheorem{conjecture}[theorem]{Conjecture}
\theoremstyle{definition}

\theoremstyle{remark}

\newtheorem{example}[theorem]{Example}

\newcommand{\A}{\mathcal{A}}

\newcommand{\R}{\mathbb{R}}
\newcommand{\Z}{\mathbb{Z}}
\newcommand{\C}{\mathcal{C}}

\title{Approval Voting in Product Societies}

\author{Kristen Mazur, Mutiara Sondjaja, Matthew Wright, Carolyn Yarnall}

\begin{document}

\begin{abstract}
In approval voting, individuals vote for all platforms that they find acceptable. In this situation it is natural to ask: When is agreement possible? What conditions guarantee that some fraction of the voters agree on even a single platform? 
Berg et.\ al.\  found such conditions when voters are asked to make a decision on a single issue that can be represented on a linear spectrum. In particular, they showed that if two out of every three voters agree on a platform, there is a platform that is acceptable to a majority of the voters.  Hardin developed an analogous result when the issue can be represented on a circular spectrum. 
We examine scenarios in which voters must make two decisions simultaneously. For example, if voters must decide on the day of the week to hold a meeting and the length of the meeting, then the space of possible options forms a cylindrical spectrum. Previous results do not apply to these multi-dimensional voting societies because a voter's preference on one issue often impacts their preference on another.  
We present a general lower bound on agreement in a two-dimensional voting society, and then examine specific results for societies whose spectra are cylinders and tori. 
\end{abstract}

\maketitle

\section{Introduction.}

Several articles in the \textsc{Monthly} have analyzed approval voting systems. The pioneering article by Berg et.\ al.\ examined agreement in \emph{linear societies}, in which voters may approve of an interval of options chosen from a linear spectrum, and higher-dimensional analogues \cite{Berg}. Recognizing that a circular model is often more appropriate, Hardin studied approval voting in \emph{circular societies} \cite{Hardin}. Furthermore, Klawe et.\ al.\ considered \emph{double-interval societies}, in which each voter may approve of two disjoint intervals in a linear spectrum \cite{Klawe}.

The prior work on approval voting theory considered primarily situations in which  voters are asked to make a decision on a single issue. However, often voters are asked to make a decision on two or more issues simultaneously. For example, suppose a group of people must decide when to hold a meeting and how long the meeting should be. Further suppose that each person is permitted to choose both an interval of possible start times and an interval of possible meeting lengths. The approval set of each voter can be represented as a rectangle in a two-dimensional spectrum, as in \cref{fig:meeting}. 

The example of \cref{fig:meeting} illustrates the importance of considering multiple choices simultaneously, rather than selecting a position of maximal approval in each choice separately. \cref{fig:meeting} displays the approval sets of five voters. Clearly, a majority of voters prefers a meeting with a later start time. A different majority prefers a shorter meeting length. If these preferences were considered independently, the chosen meeting parameters would fall in the purple rectangle, which is acceptable to only a single voter. However, considering start time and meeting length  simultaneously, then the maximal approval occurs in the areas where two voters agree.

\begin{figure}[h]
\begin{center}
    \begin{tikzpicture}
    	\draw[<->] (0,4.1) -- (0,0) -- (4.5,0);
        \node[below] at (2.25,0) {meeting start time};
        \node[left,rotate=90,anchor=south] at (0,2.05) {meeting length};
        \fill[blue,opacity=0.5] (0.5,0.3) rectangle (2,1.8);
        \fill[green,opacity=0.5] (0.7,0.5) rectangle (2.2,2);
        \fill[red,opacity=0.5] (2.5,2.5) rectangle (4,4);
        \fill[orange,opacity=0.5] (2.7,2.3) rectangle (4.2,3.8);
        \fill[violet,opacity=0.5] (3,0.8) rectangle (3.8,1.6);
    \end{tikzpicture}
    \caption{Approval regions of five voters who are asked to give approval intervals for both the start time of a meeting and the meeting length.}
    \label{fig:meeting}
\end{center}
\end{figure}
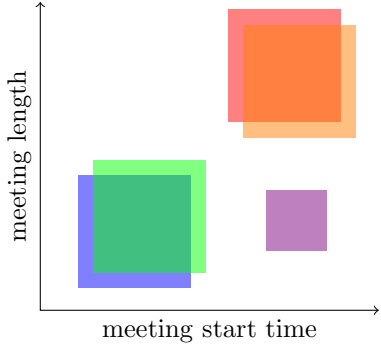

This paper examines approval voting in \emph{product societies}, in which voters are asked to make multiple decisions simultaneously. For example, suppose a committee wants to create an annual conference. They must decide on the length of the conference (in days) and what time of year to hold the conference. Since the former spectrum is linear and the latter is circular, approval sets can be modeled as rectangles on a cylinder. Thus, this society is a \textit{cylindrical society}. Or, in \cite{Hardin} Hardin gave the circular society example of a group of scouts lost in the woods needing to decide on a compass bearing. Now  suppose that these scouts must also decide on how long to walk in a given direction. Then again this is a cylindrical society. Finally, suppose an astronomy class must decide on a time of year to look at the stars and on a compass direction to point the telescope.  Both of these spectra are circular, and their product is a torus. Hence, we have created a \textit{toroidal society.}

In such product societies, we answer the following question: ``If we know voters' agreement on each choice individually, what can we say about their agreement on all issues simultaneously?''
After analyzing such product societies in general, we give more specific results for  cylindrical and torodial societies.

In the next section we review results from \cite{Berg} and \cite{Hardin}, define product societies, and discuss why the linear and circular results do not always extend to multidimensional societies. \Cref{agreement_nums} provides a general theorem about the lower bound on the maximum number of voters who agree in product societies. In \Cref{sec:cylindrical} we employ the concept of an agreement graph to give a new lower bound on the agreement of cylindrical societies. \Cref{sec:toroidal} presents a similar theorem for toroidal societies, along with the construction of a type of ``uniform" toroidal society with low overall agreement.

\section{Preliminaries.}

We use terminology from \cite{Berg} and \cite{Hardin}. A \textit{platform} is an item or candidate for which an individual can vote. We call the collection of all possible platforms a \textit{spectrum} $X$. Further, let $V$ be a finite collection of voters. Each voter $v$ in $V$ has an \textit{approval set} $A_v \subseteq X$ of platforms acceptable to the voter. We denote the collection of all approval sets of all voters in $V$ by $\mathcal{A}$. Then define a \textit{society} $S$ to be a triple $(X,V,\mathcal{A})$ consisting of a spectrum $X$, a finite collection of voters $V$ and the corresponding collection of approval sets $\mathcal{A}$. Define the \textit{size} of $S$, denoted $|S|$, to be the number of voters in $S$. The \textit{agreement number} $a(S)$ of a society $S$ is the greatest number of voters that approve a single platform, and the \textit{agreement proportion} is the ratio $\frac{a(S)}{|S|}$. Finally, let $k$ and $m$ be integers such that $1 \le k \le m \le |S|$. A society is \textit{$(k,m)$-agreeable} if $k$ out of every $m$ voters agree on some platform. 

Berg et.\ al.\ introduced the notion of $(k, m)$-agreeability to model the diversity of preferences within a \textit{linear} societies \cite{Berg}. The spectrum of a linear society is a closed subset of $\mathbb{R}$ and approval sets are closed and bounded intervals in $\mathbb{R}$.  We provide an example of a $(2,3)$-agreeable linear society with agreement number 4 in \cref{fig:linear}.

Intuitively, in a $(k,m)$-agreeable society, $k$ and $m$ parametrize the amount of overlap between voters' approval sets; the closer $k$ is to $m$, the more similar the preferences of any $m$ voters.  For example, if a society is $(k, k)$-agreeable then every group of $k$ voters unanimously approves of some platform.  

Since $(k, k)$-agreeability seems to be a very strong condition, we might predict that this would imply that there must be a platform that is approved by a large proportion of voters.  Indeed, if a linear society is $(k, k)$-agreeable with $k\geq 2$, then Helly's Theorem\footnote{More generally, Helly's Theorem implies that a $(d+1, d+1)$-agreeable society whose approval sets are convex subsets of $\mathbb{R}^d$ must contain a platform that is approved by every voter.}, a result about intersections of convex sets,  asserts that there must be a platform that is approved by \emph{every} voter.

While $(k, k)$-agreeability can lead to strong conclusions, this condition is very restrictive and may not be realistic for a typical society.  The condition of $(k, m)$-agreeability, for $k \leq m$, thus allows us to model societies that are very agreeable, very ``disagreeable,'' or something in between.  In particular, the larger $m$ is relative to $k$, the more ``disagreeable'' a $(k, m)$-agreeable society might be.  Note that if a society is $(k, m)$-agreeable, it is also $(k', m)$-agreeable for $k' \leq k$.

For example, in a $(1,m)$-agreeable society \emph{it may be possible} to find a group of $m$ voters in which none of the voters agree on a platform.  Note, however, that even if a society is $(1, m)$-agreeable, it does not necessarily have a small agreement proportion; $(1,m)$-agreeability simply asserts that among an arbitrary group of $m$ voters in the society, we cannot guarantee that there are two voters who approve of a common platform.  Thus, conclusions that we derive from assuming $(k, m)$-agreeability are only lower bounds to agreement proportions.  Nevertheless, we might intuitively predict that the lower bound on agreement proportions of such societies will be small.  The main theorem of \cite{Berg} stated below confirms this intuition for linear societies.

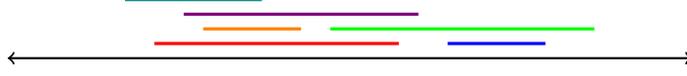
\begin{figure}[ht]
\begin{center}
	\begin{tikzpicture}[scale=1.3]
        \draw[<->,thick] (-1,0) -- (6,0);
        \draw[red, very thick] (0.5,0.15) -- (3,0.15);
    	\draw[blue, very thick] (3.5,0.15) -- (4.5,0.15);
    	\draw[orange, very thick] (1,0.3) -- (2,0.3);
    	\draw[green, very thick] (2.3,0.3) -- (5,0.3);
    	\draw[violet, very thick] (0.8,0.45) -- (3.2,0.45);
    	\draw[teal, very thick] (0.2,0.6) -- (1.6,0.6);
    \end{tikzpicture}
    \end{center}
	\caption{A $(2,3)$-agreeable linear society with six voters and agreement number $4$. The black line is the (infinite) spectrum, and each colored segment represents the approval set of one voter.}
    \label{fig:linear}
\end{figure}

\begin{theorem}[Agreeable Linear Society Theorem \cite{Berg}]\label{thm:linear}
Let $2 \le k \le m$. In a $(k,m)$-agreeable linear society $S$, the best possible lower bound for the agreement proportion of $S$ is $(k-1)/(m-1)$. 
\end{theorem}
In other words, if $S$ is a $(k,m)$-agreeable linear society, then $\frac{a(S)}{|S|} \ge (k-1)/(m-1),$ and there exists a $(k,m)$-agreeable linear society such that the agreement proportion equals $(k-1)/(m-1)$.

Hardin proved a similar theorem for circular societies \cite{Hardin}. A \textit{circular} society is a society in which the spectrum is a circle and the approval sets are connected subsets of the circle. Thus, approval sets are typically arcs. We give an example of a $(2,2)$-agreeable circular society with agreement number 2 in \cref{fig:circular}. Hardin's main result in \cite{Hardin} is the following lower bound on the agreement proportion of circular societies.

\begin{figure}[ht]
    \begin{center}
    \begin{tikzpicture}
        \draw[thick] (0,0) circle (1cm);
        \draw[blue, very thick] (0:1.15) arc (0:200:1.15cm);
        \draw[green, very thick] (120:1.3) arc (120:320:1.3cm);
        \draw[red, very thick] (240:1.45) arc (240:440:1.45cm);
    \end{tikzpicture}

	\caption{A $(2,2)$-agreeable circular society with three voters and agreement number $2$. The black circle is the spectrum, and colored arcs represent approval sets.}
    \label{fig:circular}
        \end{center}
\end{figure}
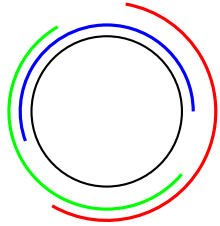

\begin{theorem}[Agreeable Circular Society Theorem \cite{Hardin}]\label{thm:circular}
Let $2 \le k \le m$. If $S$ is a $(k,m)$-agreeable circular society, then $$\frac{a(S)}{|S|} > \frac{k-1}{m},$$ and there exists a $(k,m)$-agreeable circular society such that $a(S) = \lfloor \frac{k-1}{m}|S|\rfloor + 1$.\end{theorem}

A product society is, in essence, the Cartesian product of two societies. More specifically, let $(X, V, \mathcal{A}^X)$ and $(Y, V, \mathcal{A}^Y)$ be societies that share the same collection of voters $V$. In the product of these two societies, voters from the set $V$ approve of platforms on the two independent spectra and the approval set of a single voter is the product of their approval sets from $X$ and $Y$. Thus, a {\emph{product society}} is the triple $(X\times Y, V, \mathcal{A}^{X\times Y})$, where $\mathcal{A}^{X\times Y} = \{A^X_v \times A^Y_v \mid v \in V \}$.

If $S=(X \times Y, V, \mathcal{A}^{X \times Y})$ is a product society in which both spectra $X$ and $Y$ are linear, then $S$ is a \textit{2-box} society. Its spectrum $X \times Y$ is $\mathbb{R}^2$, and the approval sets $A_v^X \times A_v^Y$ of voters in $V$ are rectangles. We note that one may wish to study a society whose spectrum of platforms is a finite rectangle. In this case, we could still consider the collection of approval sets as subsets of the larger space, $\R^2$. Berg et.\ al.\ first defined a 2-box society in \cite{Berg}. In fact, for $d \ge 1$ they generalize the notion of a 2-box society to that of a \textit{$d$-box} society whose spectrum is $\mathbb{R}^d$. Approval sets in $d$-box societies are then boxes in $\mathbb{R}^d$. Theorem 13 of \cite{Berg} gives a lower bound on the agreement number of a $(k,m)$-agreeable $d$-box society. However, we will just state the theorem for a 2-box society. 

\begin{theorem}[Agreeable 2-Box Society Theorem \cite{Berg}]\label{2box} Let $m, k \ge 2$ be integers with $k \le m \le 2k-2$, and let $S$ be a $(k,m)$-agreeable 2-box society. Then the agreement number of $S$ is at least $|S| -m+k$, and this bound is best possible.
\end{theorem}

A product society in which one spectrum is linear and the other is circular is a \textit{cylindrical} society. A product society in which both spectra are circular is a \textit{toroidal} society. Neither cylindrical nor toroidal societies have yet been studied in the literature. Moreover, just as \cref{thm:linear} does not extend to circular societies, \cref{2box} does not extend to cylindrical and toroidal societies.

Furthermore, neither cylindrical nor toroidal societies behave exactly like circular societies. Indeed, we  construct below cylindrical and toroidal societies for which the inequality given in \cref{thm:circular} does not provide a lower bound on agreement. Define a \textit{stratified} society to be a product society whose spectrum is partitioned into disjoint regions called \textit{strata} such that each voter's approval set is completely contained in a single stratum.

\begin{example}\label{ex:stratified}
	Let $S$ be a cylindrical or toroidal stratified society consisting of 20 voters whose approval sets lie in four strata. Further, suppose that each stratum contains five pairwise-intersecting approval sets such that the agreement number in each stratum is 3. In Figure \ref{fig:stratified} we give an example of such a society. By the pigeon-hole principle, $S$ is $(2,5)$-agreeable. If \cref{thm:circular} applied to $S$, then the agreement number would be at least four. However, the agreement number of this society is in fact three. Therefore, an understanding of cylindrical and toroidal societies requires further exploration.
\end{example}

\begin{figure}[htb]
\begin{center}
   	 \begin{tikzpicture}[scale=0.6]
    	\def\a{1}	
        \def\b{2}	
        \def\t{0.5} 
        \def\s{2.5} 
 
        \foreach \i in {1,2,3,4}
            \draw[thick,dashed,gray] (\s*\i, \b) arc(90:270:{\a} and {\b});
        
        \fill[lightgray,fill opacity=0.5] (\s,\b) arc(90:-90:{\a} and {\b}) -- (4*\s,-\b) arc(-90:90:{\a} and {\b}) -- cycle;
        \fill[lightgray,fill opacity=0.3] (\s,\b) arc(90:270:{\a} and {\b}) -- (4*\s,-\b) arc(270:90:{\a} and {\b}) -- cycle;
        
		\draw[thick] (0,0) ellipse ({\a} and \b);
        \foreach \i in {1,2,3,4}
            \draw[thick] (\s*\i,-1*\b) arc(-90:90:{\a} and {\b});
        \draw[thick] (0,-1*\b) -- (10,-1*\b);
        \draw[thick] (0,\b) -- (10,\b);
        
        \draw [decorate,decoration={brace,amplitude=8pt},xshift=0pt,yshift=-6pt]
(5*\t,-\b) -- (0,-\b) node [black,midway,yshift=-14pt] {one strata};
\end{tikzpicture}
\hspace{1cm}
\begin{tikzpicture}
	\draw[thick] (0,0) circle (1cm);
    \draw[blue, very thick] (0:1.15) arc (0:200:1.15cm);
    \draw[green, very thick] (72:1.3) arc (72:272:1.3cm);
    \draw[yellow, very thick] (144:1.45) arc (144:344:1.45cm);
    \draw[orange, very thick] (216:1.6) arc (216:416:1.6cm);
    \draw[red, very thick] (288:1.75) arc (288:488:1.75cm);
\end{tikzpicture}
    \caption{A cylindrical stratified society containing four strata (left). Each stratum contains five pairwise-intersecting approval sets with cross section as shown (right).}
    \label{fig:stratified}
\end{center}
\end{figure}
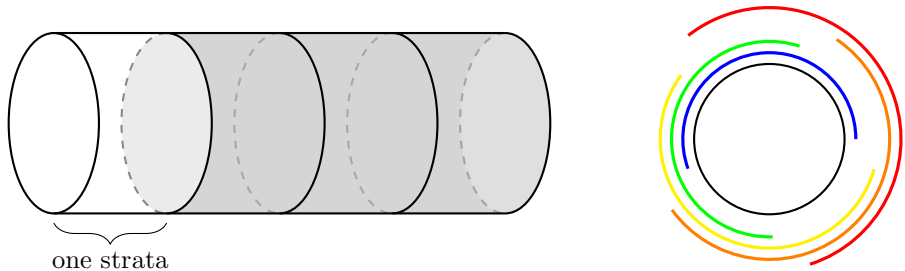

\section{Agreement Numbers in Product Societies.}\label{agreement_nums}

In this section, we provide a lower bound for the agreement number of $(k,m)$-agreeable product societies. We then analyze the strength of the bound. 

\subsection{A bound for the agreement number.} We begin by examining certain projections within a product society.
Let $S = (X\times Y, V, \mathcal{A}^{X \times Y})$ be a product society. Let $S^X$ be the projection of $S$ onto the spectrum $X$ and $S^Y$ the projection onto $Y$; that is, $S^X = (X, V, \mathcal{A}^X)$ and $S^Y = (Y,V, \mathcal{A}^Y)$. Note that $|S^X| = |S^Y| = |S|$. Further, let $p^X$ be a platform with maximal agreement in $S^X$. In other words, $p^X$ is the location in $X$ of the agreement number $a(S^X)$. Define $p^Y$ analogously. Let $\widetilde{S}_{p^X}$ be the subsociety of $S$ that consists of the subset of voters that approve of $p^X$. So, $|\widetilde{S}_{p^X}| = a(S^X)$. Then let $\widetilde{S}^Y_{p^X}$ be the projection of $\widetilde{S}_{p^X}$ onto the spectrum $Y$. We define $\widetilde{S}_{p^Y}^X$ analogously.

\begin{lemma}\label{genlem} Let $S=(X\times Y, V, \A^{X \times Y})$ be a product society. Then
\[ a(S) \ge a(\widetilde{S}^Y_{p^X}) \quad \text{and} \quad a(S) \ge  a(\widetilde{S}^X_{p^Y}). \]
\end{lemma}

\begin{proof} First, in order for approval sets in $S$ to overlap, their projections must intersect in both $S^X$ and $S^Y$. Without loss of generality, consider $\widetilde{S}^Y_{p^X}$. This society contains approval sets in $S^Y$ whose corresponding approval sets in $S$ project to intersecting sets in $S^X$. Thus, if approval sets intersect in $\widetilde{S}^Y_{p^X}$, then the corresponding approval sets in $S$ also intersect. Hence, $a(S) \ge a(\widetilde{S}^Y_{p^X})$.
\end{proof}

We are now ready to state our lower bound on the agreement proportion of a $(k,m)$-agreeable product society.
The following theorem relies on the fact that if $S=(X \times Y, V, \A^{X \times Y})$ is a $(k,m)$-agreeable product society, then the projections $S^X$ and $S^Y$ are $(k,m)$-agreeable as well.

\begin{theorem}\label{general_theorem}
Let $S = (X \times Y, V, \A^{X \times Y})$ be a $(k,m)$-agreeable product society. Let $\alpha$ be the lower bound on the agreement proportion of $S^X$ imposed from the fact that $S^X$ is $(k,m)$-agreeable. Let $\beta$ be the analogous lower bound on the agreement proportion of $S^Y$. If $|S| \ge \frac{m}{\min(\alpha,\beta)}$, then 
\[ \frac{a(S)}{|S|} \ge \alpha \beta. \]
\end{theorem}

\begin{proof} Without loss of generality, let $\alpha \le \beta$. By \cref{genlem}, $a(S) \ge a(\widetilde{S}^Y_{p^X})$. Thus, we will put a bound on $a(\widetilde{S}^Y_{p^X})$. The society $S^X$ is at least $(k,m)$-agreeable, so $a(S^X) \ge |S^X|\alpha = |S| \alpha$. Since $|S| \ge \frac{m}{\alpha}$, it follows that $a(S^X) \ge m$. Thus, the society $\widetilde{S}_{p^X}$ contains at least $m$ voters and is still $(k,m)$-agreeable. So $\widetilde{S}^Y_{p^X}$ is also at least $(k,m)$-agreeable, and hence $a(\widetilde{S}^Y_{p^X}) \ge |\widetilde{S}^Y_{p^X}|\beta$. Finally, $|\widetilde{S}^Y_{p^X}| = a(S^X)$, and so $a(\widetilde{S}^Y_{p^X}) \ge |S| \alpha \beta$. Hence, $\frac{a(S)}{|S|} \ge \alpha \beta$.
\end{proof}

If we apply \cref{general_theorem} to 2-box, cylindrical, and toroidal societies, then we obtain the following bounds on agreement.

\begin{theorem}\label{main_cor} Let $S$ be a $(k,m)$-agreeable product society.
\begin{itemize}
\item If $S$ is a 2-box society and $|S| \ge \frac{m(m-1)}{k-1}$, then 
	\[ \frac{a(S)}{|S|} \ge  \frac{(k-1)^2}{(m-1)^2}. \]
\item If $S$ is a cylindrical society and $|S| \ge \frac{m(m-1)}{k-1}$, then
	\[ \frac{a(S)}{|S|} >  \frac{(k-1)^2}{m(m-1)}. \]
\item If $S$ is a toroidal society and $|S| \ge \frac{m^2}{k-1}$, then 
	\[ \frac{a(S)}{|S|} >  \frac{(k-1)^2}{m^2}. \]
\end{itemize}
\end{theorem}

\begin{proof}
Let $S$ be a $(k,m)$-agreeable 2-box society with $|S| \ge \frac{m(m-1)}{k-1}$. Then both $S^X$ and $S^Y$ are  $(k,m)$-agreeable linear societies with agreement proportions  at least $\frac{k-1}{m-1}$. 
Thus, in the notation of \cref{general_theorem}, we have $\alpha = \beta = \frac{k-1}{m-1}$.  Since $\frac{m}{\min(\alpha,\beta)} = \frac{m(m-1)}{k-1}$ and $|S| \ge \frac{m(m-1)}{k-1},$ it follows that $\frac{a(S)}{|S|} \ge \alpha\beta = \frac{(k-1)^2}{(m-1)^2}$.

The proofs for cylindrical and toroidal societies are analogous.
\end{proof}

For example, in a $(3,3)$-agreeable cylindrical society with a sufficiently large number of voters, at least two-thirds of the voters agree on some platform. On the other hand, a toroidal society must be $(6,6)$-agreeable for us to conclude that at least two-thirds of the voters agree on some platform.

Moreover, it may be much easier to analyze voting preferences of small groups instead of trying to understand all preferences at once, in which case the result above is useful. For instance, suppose an organization is trying to determine the length of a conference (in days) and what time of year to hold the conference, and all members of the organization vote for their preferred length and date combination. As mentioned in the Introduction, this is an example of a cylindrical society. The organization wants to choose a length and date combination that works for at least two-thirds of the voting members.  Since it may be difficult to analyze all approval sets at once, the organization takes samples of three voters and determines that in each sample, the three voters agree on a length and date. While it may be impractical to consider all possible samples of three voters, based on the samples collected, the organization conjectures that the society is $(3,3)$-agreeable. Based on the hypothesis of $(3,3)$-agreeability, the organization invokes \cref{main_cor} to conclude that there is a meeting length and date that is acceptable to at least two-thirds of the members.

\subsection{Analyzing bounds.} We now discuss the strength of the result from \cref{main_cor} by comparing to previous results of Berg, et. al. and considering examples in which the bound is sharp.

We first consider $(k,m)$-agreeable $2$-box societies.  \cref{2box} gives a stronger bound for $2$-box societies than  \cref{main_cor}. However, \cref{2box} only holds under strict assumptions about $k$ and $m$ since it requires $k\leq m \leq 2k-2$. The result from \cref{main_cor} has no restrictions on $k$ or $m$ and thus we may apply it to more $2$-box societies. While an assumption must be made about the size of the society in order to apply \cref{main_cor}, it only becomes very restrictive when $m$ is significantly larger than $k$, in which case $a(S)$ must be small anyway.

Cylindrical and toroidal societies will be discussed in greater detail in the following two sections but first, we provide examples demonstrating that the result from \cref{main_cor} is sharp in certain cases.

\begin{example}\label{cylinder_ex}
Let $S$ be the $(2,2)$-agreeable cylindrical society consisting of $3$ approval sets that pairwise intersect around the cylinder in such a way that not all three intersect at any point. A cross section of this society is shown in \cref{fig:circular}.
Then $|S| = 3$ which is greater than $\frac{2(2-1)}{2-1}$, so we may apply \cref{main_cor}. The result is that $a(S) > \frac{1}{2}|S|$, giving a lower bound of $2$ on the agreement number, which is the exact value of $a(S)$.
\end{example}

\begin{example}
Let $S$ be the $(2,2)$-agreeable toroidal society with $5$ voters shown in \cref{fig:torus5}. To construct the torus, we identify the sides as indicated. Since $|S| > \frac{2^2}{2-1}$, \cref{main_cor} gives $a(S) > \frac{5}{4}$. Thus, the result gives a lower bound of $2$ which is precisely the agreement number of the society.
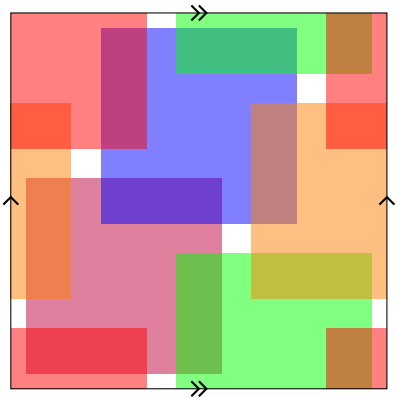
\begin{figure}[h]
\begin{center}
	\begin{tikzpicture}
    \fill[purple,opacity=0.5] (0.2,0.2) rectangle (2.8,2.8);
    \fill[blue,opacity=0.5] (1.2,2.2) rectangle (3.8,4.8);
    \fill[green,opacity=0.5] (2.2,0) rectangle (4.8,1.8);
    \fill[green,opacity=0.5] (2.2,4.2) rectangle (4.8,5);
    \fill[orange,opacity=0.5] (3.2,1.2) rectangle (5,3.8);
    \fill[orange,opacity=0.5] (0,1.2) rectangle (0.8,3.8);
    \fill[red,opacity=0.5] (4.2,3.2) rectangle (5,5);
    \fill[red,opacity=0.5] (0,0) rectangle (1.8,0.8);
    \fill[red,opacity=0.5] (0,3.2) rectangle (1.8,5);
    \fill[red,opacity=0.5] (4.2,0) rectangle (5,0.8);
   	
    \draw (0,0) rectangle (5,5);
    \draw[thick] (2.4,0.1) -- (2.5,0) -- (2.4,-0.1);
    \draw[thick] (2.5,0.1) -- (2.6,0) -- (2.5,-0.1);
    \draw[thick] (2.4,5.1) -- (2.5,5) -- (2.4,4.9);
    \draw[thick] (2.5,5.1) -- (2.6,5) -- (2.5,4.9);
    \draw[thick] (-0.1,2.45) --(0,2.55) -- (0.1,2.45);
    \draw[thick] (4.9,2.45) -- (5,2.55) -- (5.1,2.45);
    \end{tikzpicture}
\end{center}
\caption{A $(2,2)$-agreeable toroidal society with $5$ voters and agreement number $2$.}
\label{fig:torus5}
\end{figure}

\end{example}

\section{Cylindrical Societies.}\label{sec:cylindrical}

Before presenting our main theorem for this section, which gives a lower bound on the agreement number of a $(k,m)$-agreeable cylindrical society, we first revisit the agreement number of a $(k,m)$-agreeable circular society. \cref{thm:circular} gives Hardin's lower bound for the agreement of a $(k,m)$-agreeable circular society in terms of the size of the society. His method for proving this bound relies on transforming a given circular society to one of a particular form called a {\it uniform society}. We take a different approach involving the agreement graph of a society. 

Recall that a graph $G$ consists of a collection of vertices $V(G)$ and edges $E(G)$ that join pairs of vertices. The {\it agreement graph} $G$ of a society $S$ is constructed by letting $V(G)$ consist of all voters in the society and forming edges between vertices when the corresponding voters have intersecting approval sets. The results in this section rely on the clique number of agreement graphs. A {\it clique} in a graph is a collection of vertices such that each pair of vertices share an edge. The maximum clique of a graph $G$ is a clique with the largest number of vertices. The {\it clique number} of $G$, denoted $\omega(G)$, is the number of vertices in a maximum clique of $G$.
In an agreement graph, the clique number gives the maximal number of voters that are pairwise-agreeable.

\subsection{Revisiting circular societies.}

For a linear society, the clique number of the agreement graph is equal to the agreement number of the society \cite[Fact 1]{Berg}. The same is not true for circular societies. In fact, the clique number can be larger than the agreement number---consider a circle with three arcs  that intersect pairwise, but with no point in common to all three. We can, however, give a lower bound on the agreement number of a circular society in terms of the clique number of its agreement graph. 

\begin{theorem}\label{circ_clique}
	Let $S$ be a circular society, and let $G$ be the agreement graph of $S$. Then,
    \begin{equation*}
    	a(S) \ge \frac{\omega(G) + 1}{2}.
    \end{equation*}
\end{theorem}

\begin{proof} For any circular society $S$ the agreement graph $G$ has some maximal clique $G_M$ of size $\omega(G)$. The voters corresponding to the vertices of this graph form a pairwise-agreeable sub-society $\bar{S} \subseteq S$, and so $\bar{S}$ is $(2,2)$-agreeable. Moreover, since $a(S) \ge a(\bar{S})$, it suffices to show that $$a(\bar{S}) \ge \frac{\omega(G)+1}{2}.$$ Since $|\bar{S}| = \omega(G)$, we will show that $|\bar{S}| \le 2a(\bar{S}) -1$.

We will use a construction that transforms the circular society $\bar{S}$ into a linear society via a process of cutting and ``unrolling." Let $p$ be a point (i.e., a representative platform) in the intersection of $a(\bar{S})=a$ approval sets. We can cut $\bar{S}$ at the point $p$ and ``unroll" to form a linear spectrum. In doing so, we  keep a copy of $p$, call it $p'$, at the other end of the newly-constructed linear society and consider the two pieces resulting from cutting the $a$ intersecting approval sets as unique sets.
In this way, we have created a linear society $S_L$ with $|\bar{S}| + a$ approval sets such that $a(S_L) = a(\bar{S})$. See \cref{circular to linear} for an example.

\begin{figure}[h]
\begin{tikzpicture}[scale=.9, transform shape]
	\draw[thick] (0,0) circle (1cm);
    \draw[blue, very thick] (1.75,0) arc (0:150:1.75cm);
    \draw[orange, very thick] (0,-1.3) arc (-90:60:1.3cm);
    \draw[green, very thick] (-.725,1.256) arc (-240:-20:1.45cm);
    \draw[yellow, very thick] (-.547,-1.5) arc (-110: 30: 1.6cm);
    \draw[red, very thick] (0, 1.15) arc (90: -55: 1.15cm);
    \draw[dashed] (.6,.6) -- (1.4,1.4);
    \node at (.45,.45) {$p$};
    \node at (0, -1.85) {$\bar{S}$};
\end{tikzpicture}
\hspace{.9cm}
\begin{tikzpicture}[scale=1.2]
    \draw[thick,<->] (0,0) -- (7,0);
    \draw[dashed] (0.5,-.1) -- (0.5,.9);
    \draw[dashed] (6.5,-.1) -- (6.5,.9);
    \draw[blue, very thick](.5,.75) -- (1.3,.75);
    \draw[orange,very thick](.5,.3) -- (2.5,.3);
    \draw[red, very thick] (0.5,0.15) -- (2,0.15);
    \draw[green,very thick] (1.7,.45) -- (5,.45);
   \draw[red, very thick](5.5,.15) -- (6.5,.15);
   \draw[yellow, very thick](1,.6) -- (3, .6);
   \draw[blue, very thick](4.5, .75) -- (6.5,.75);
   \draw[orange,very thick] (6, .3) -- (6.5,.3);
   \node at (.5, -.25) {$p$};
   \node at (6.5, -.25) {$p'$};
   \node at (3.5, -.5) {$S_L$};
\end{tikzpicture}
\caption{We cut the spectrum of a circular society $\bar{S}$ at a point $p$ to transform $\bar{S}$ into a linear society $S_L$ with $|S_L| = |\bar{S}| + a(\bar{S})$.}
\label{circular to linear}
\end{figure}
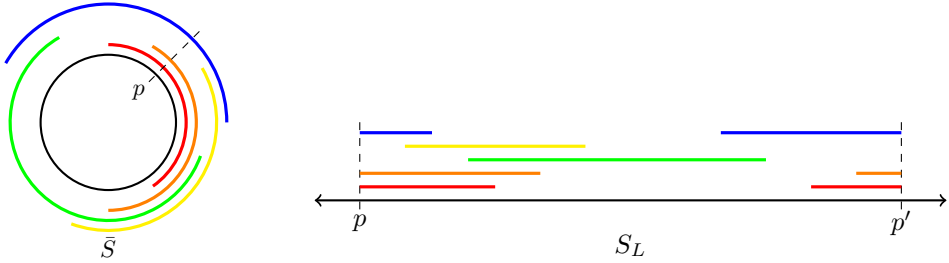

The agreement graph $G_L$ of $S_L$ has $|\bar{S}|+a$ vertices, which we label
\[ v_1, v_2, \dots, v_a, v_1', v_2', \dots, v_a', w_1, w_2, \dots, w_{|\bar{S}|-a}, \]
where the $v_i$ and $v_i'$ are the vertices corresponding to the two pieces of each of the $a$ intersecting approval sets that we cut in the construction of $S_L$, and the $w_j$ correspond to the remaining $|\bar{S}|-a$ approval sets. 
More precisely, the vertices $v_i$ correspond to the approval sets in $S_L$ containing $p$, $v_i'$ correspond to the sets containing $p'$, and $w_j$ correspond to sets that do not contain $p$ or $p'$.

Since the approval sets associated to the collection $v_1, \dots, v_a$ are all intersecting, we know that the collection of these vertices in $G_L$ form a clique of size $a$. Likewise, the vertices $v_1', \dots, v_a'$ also form a clique of size $a$. Additionally, since $\bar{S}$ is $(2,2)$-agreeable, the approval sets corresponding to the vertices $w_1, \dots w_{|\bar{S}|-a}$ are all pairwise intersecting, so the collection of $w_j$ forms a clique of size $|\bar{S}|-a$ in $G_L$.

To obtain the graph $G_M$ from the graph $G_L$, we collapse each pair of vertices $v_i$ and $v_i'$ to a single vertex, and if any $w_j$ has an edge to both $v_i$ and $v_i'$, we remove one of those edges. Since $G_M$ is the complete graph on $|\bar{S}|$ vertices, we know that each $w_j$ must share an edge with either $v_i$ or $v_i'$ for each $i$.

Without loss of generality, we can order the $v_i$ so that $v_1$ corresponds to the shortest of the approval sets containing $p$ and $v_a$ corresponds to the longest. Thus, if any $w_j$ shares an edge with a fixed $v_k$, it must also share an edge with every $v_i$ for $k \leq i \leq a$. 

In particular, each $w_j$ must either share an edge with $v_a$ or $v_i'$ for all $i$. But, if the latter case were true for some $w_j$, then $w_j \cup \{v_i'\}_{i=1}^{a}$ would form a clique of size $a+1$, and for a linear society, the clique number of the agreement graph equals the agreement number of the society \cite[Fact 1]{Berg}. Thus, a clique of size $a+1$ in $G_L$ would imply that $a(S_L) \ge a+1$, which contradicts the fact that $a(S_L) = a(\bar{S}) = a$.

Therefore, every $w_j$ must share an edge with $v_a$. This implies that the number of $w_j$ must be less than $a$, for otherwise we would again have a clique of size $a+1$. Hence, $|\bar{S}|-a < a$, or equivalently, $|\bar{S}| \leq 2a -1$. Thus, $\omega(G) \ge 2a(\bar{S}) -1$, and it follows that $a(S) \ge (\omega(G)+1)/2$.

\end{proof}

\subsection{A new bound for cylindrical societies.}

We now turn to cylindrical societies. We will make use of the following theorem from Berg, et. al.\ \cite[Theorem 11]{Berg} regarding the clique size of a graph.

\begin{theorem}\label{berg_thm11}
Let $m,k \geq 2$ be integers with $k\leq m \leq 2k-2$, and let $G$ be a graph on
$n \geq m$ vertices such that every subset of $m$ vertices includes a clique of size $k$.
Then $\omega(G)\geq n-m+k$.
\end{theorem}

For the same reason as in the circular society case, the clique number of the agreement graph of a cylindrical society is not the same as the agreement number of the society. Thus, the theorem above is not enough to immediately give a result regarding the agreement number of a cylindrical society. However, we may use Berg's result in conjunction with \cref{circ_clique} to prove the following result. It is the cylindrical society version of \cref{2box}, Berg's result for the agreement number of $2$-box societies.

\begin{theorem}\label{cylinder_bound} 
	Let $k$ and $m$ be integers such that $2 \le k \le m \le 2k-2$. If $S $ is a $(k,m)$-agreeable cylindrical society, then
	\begin{equation*}
    	a(S) \geq \frac{|S|-m+k+1}{2}.
    \end{equation*}
\end{theorem}

\begin{proof}
	 Let $S = (X\times Y, V, \A^{X \times Y})$, and let $G$ be the agreement graph of $S$; by \cref{berg_thm11} there exists a clique $G'$ of size $|S|-m+k$ in $G$.
    Let $\A' \subseteq \A^{X \times Y}$ be the collection of approval sets that correspond to the vertices in $G'$. That is, $\A'$ consists of $|S|-m+k$ pairwise-intersecting approval sets in $X \times Y$.
    
    Since $S$ is a cylindrical society, without loss of generality, let $X$ be a linear spectrum and $Y$ be a circular spectrum. Let $\pi : X \times Y \to X$ be the canonical projection.
    Then the set $\{ \pi(A) \mid A \in \A' \}$ is a set of $|S|-m+k$ pairwise-intersecting intervals in $X$.
    By Helly's Theorem, the intersection
    \[ \bigcap_{A \in \A'} \pi(A) \]
    is nonempty, so there exists some point $x$ in this intersection.
    Then $C=\pi^{-1}(x)$ is a circular spectrum in $X \times Y$ that intersects every $A \in \A'$. 
    
	Let $V_C \subseteq V$ be the subset of voters whose approval sets intersect $C$, and let $\A_C$ be the corresponding collection of intersections of their approval sets with $C$. Then $S_C=(C,V_C,\A_C)$ is a circular society that is the ``slice'' of $S$ along $C$. By construction of $S_C$, its agreement graph has a clique of size $|S|-m+k$. By \cref{circ_clique}, $a(S_C) \geq \frac{|S|-m+k+1}{2}$. Since $a(S) \ge a(S_C)$, the proof is complete.
\end{proof}

As an application of \cref{cylinder_bound}, consider the scouts mentioned in the Introduction that are lost in the woods. They must decide on a compass direction to travel and how long to travel in a given direction. This cylindrical society does not want to move in any direction unless the majority agree on a direction and travel time. With their knowledge of \cref{cylinder_bound}, they know that if among any $4$ of them, at least $3$ agree on a direction and time, then there will be a direction and time combination that is acceptable to the majority of the group. Conjecturing $(3,4)$-agreeability, they find this result reassuring, even though the \emph{existence} of agreement does not necessarily help them get out of the woods. Note that if there are at least $6$ scouts and they form a $(3,4)$-agreeable society, then both \cref{main_cor} and \cref{cylinder_bound} apply in this situation, but \cref{cylinder_bound} gives a stronger result.

\subsection{Comparing results.} We now compare the bound for the agreement number of cylindrical societies from \cref{cylinder_bound} to the bound for cylindrical societies from \cref{main_cor}. We then provide examples demonstrating that the bound from \cref{cylinder_bound} is sharp in some cases.

We first note that  \cref{main_cor} and \cref{cylinder_bound} hold under different assumptions. \cref{main_cor} holds only for sufficiently large $(k,m)$-agreeable societies ($|S| > \frac{m(m-1)}{k(k-1)}$), while \cref{cylinder_bound} holds when $k$ and $m$ are close ($k \leq m \leq 2k-2$). Thus, each result's strength depends on the size of societies we wish to study or on how agreeable they are. In order to compare the two results we now only consider cylindrical societies that satisfy both of the aforementioned assumptions. We now consider two examples which correspond to the extreme cases from the bound $k \leq m \leq 2k-2$.

\begin{example}
Let $S$ be a $(k,k)$-agreeable cylindrical society. Applying the bound from \cref{main_cor} yields 
\[ a(S) > \frac{k-1}{k}|S|, \]
while the result from \cref{cylinder_bound} yields
\[ a(S) \geq \frac{|S|+1}{2}. \]
\end{example}

Hence, whenever $k \geq 3$, \cref{main_cor} is a stronger result in the case above. However, as we will see in the example below, when $m$ is large compared to $k$,  \cref{cylinder_bound} will sometimes give a better bound. 

\begin{example}
Let $S$ be a $(k,2k-2)$-agreeable society with $|S| > 4$. Thus, we may apply \cref{main_cor} to obtain
\[a(S) > \frac{k-1}{2(2k-3)}|S|.\]
On the other hand, applying \cref{cylinder_bound} gives
\[a(S) \geq \frac{|S| - k + 3}{2}.\]
\end{example}

For small values of $k$ the bound from \cref{cylinder_bound} is stronger in the example above, but as the value of $k$  increases, the bound for $a(S)$ from \cref{main_cor} approaches $\frac{|S|}{4}$. In particular, the bound from \cref{cylinder_bound} is still stronger unless $k > |S|/2$. However, in the case $k > |S|/2$, we would have $m > |S|-2$ and thus would already know almost entirely the agreement number of the society.

Finally, the bound from \cref{cylinder_bound} is sharp for the society in \cref{cylinder_ex}. Since the society is $(2,2)$-agreeable, \cref{cylinder_bound} holds and gives  $a(S) \geq 2$.

\section{Toroidal Societies.}\label{sec:toroidal}

We conclude with another look at toroidal societies. We first give a result analogous to Theorem \ref{cylinder_bound} that holds regardless of $|S|$, unlike \cref{main_cor}. We then show that the toroidal society bounds in Theorems \ref{main_cor} and \ref{tor_bound} are sharp for a specific family of toroidal societies.

\begin{theorem}\label{tor_bound}
Let $k$ and $m$ be integers such that $2 \le k \le m \le 2k-2$. If $S = (X\times Y, V, \A^{X \times Y})$ is a $(k,m)$-agreeable toroidal society, then
	\begin{equation*}
    	a(S) \geq \frac{|S|-m+k+2}{4}.
    \end{equation*}
\end{theorem}

\begin{proof}
This proof is analogous to the proof of \cref{cylinder_bound}, except now the set $\{\pi(A) \mid A \in \A'\}$ consists of $|S|-m+k$ pairwise-intersecting intervals in the \textit{circular} society $X$. Hence, this set forms a $(2,2)$-agreeable circular society. Helly's theorem no longer applies. However, by \cref{thm:circular} there exists some platform $x$ in the intersection of at least $\frac{|S|-m+k}{2}$ of these sets. We then define $S_C$ as we did in the proof of \cref{cylinder_bound}, and by \cref{circ_clique}, $a(S_C) \geq \frac{\frac{|S|-m+k}{2}+1}{2} = \frac{|S| - m + k + 2}{4}$. Therefore, $a(S) \ge \frac{|S| - m + k + 2}{4}$ as well.
\end{proof}

The astronomy class mentioned in the Introduction forms a toroidal society by attempting to decide on a time of year to view the stars and a compass direction to point their telescope. 
They realize that even if they form a $(3,4)$-agreeable society $S$, then \cref{tor_bound} only guarantees that $\frac{|S|+1}{4}$ class members will agree on a time and direction combination. While the members of the class are troubled by their {\it potential} disagreement, they are heartened by the fact the theorem only gives a lower bound.

The previous example begs the question: is the lower bound in \cref{tor_bound} sharp? 
We have found sharpness questions in product societies difficult to answer in general, but for $(2, 2)$-agreeable toroidal societies, we are able to show that the bounds in Theorems \ref{main_cor} and \ref{tor_bound} are sharp in certain cases.  Specifically, we construct and analyze an analogue of Hardin's \emph{uniform circular societies} \cite{Hardin} that we will call \emph{uniform toroidal societies}.

\begin{example}[Uniform Toroidal Society] \label{ex:uniformToroidal22}
We construct a family of $(2, 2)$-agreeable \emph{uniform toroidal societies} whose agreement number is equal to $\left\lceil \frac{|S|-1}{4} \right\rceil + 1$.
Because we refer to $|S|$ many times in the following, let $n = |S|$.
These toroidal societies are ``uniform'' in the sense that their approval sets are congruent squares with regular intersection patterns.
We explain the construction in detail for $n \equiv 1 \pmod{4}$, which is the case that exhibits the most symmetry, and then explain briefly how to generalize to other $n$.

Let $n = 4\ell + 1$ for some positive integer $\ell$.
We construct a society $S$ whose spectrum is the quotient space $\R^2/(2n\Z \times 2n\Z)$; this quotient is a torus, with the particular denominator chosen to avoid fractions in the coordinates of approval set corners.
For concreteness, we describe the $n$ approval sets geometrically in the square $X = [0,2n) \times [0,2n)$, which we identify with the quotient space above.

Let $\Gamma$ be a path formed by the intersection of the lines $y=2x$ and $y=2x-2n$ with $X$;
this path is illustrated by the yellow lines in \cref{torus_construction}.
When opposite sides of the square $X$ are identified to form a torus, $\Gamma$ is a closed path on the torus. 
We use $\Gamma$ to specify the relative ordering of the corner points of the rectangular approval sets on the torus, just as Hardin's construction depends on the relative ordering of the endpoints of the approval sets, which are intervals on a circle.

Each of the approval sets $A_0, A_1, \ldots, A_{n-1}$ is an $n \times n$ square, with sides parallel to the  sides of $X$, whose lower-left corners are located at evenly-spaced points along $\Gamma$.
\Cref{torus_construction} illustrates the $n=9$ case, marking these points by yellow dots.
Specifically, for $i \in \{0, 1, \ldots, n-1 \}$, the lower-left corner of $A_i$ is located at $(2i,4i)$, with the second coordinate reduced modulo $2n$.

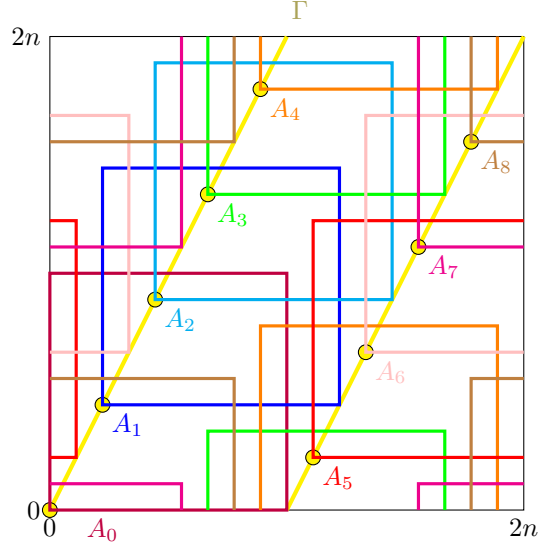
\begin{figure}[h]
\begin{center}
\begin{tikzpicture}[scale=0.35]
    \draw (0,0) rectangle (18,18);
    
    \draw[ultra thick,yellow] (0,0) -- (9,18);
    \draw[ultra thick,yellow] (9,0) -- (18,18);
    
    \foreach \i in {0,2,...,8}
    	\fill[yellow] (\i, 2*\i) circle (8pt);
    \foreach \i in {0,2,...,8}
        \draw[thin] (\i,2*\i) circle (8pt);
    \foreach \i in {10,12,14,16}
    	\fill[yellow] (\i, 2*\i-18) circle (8pt);
    \foreach \i in {10,12,14,16}
        \draw (\i,2*\i-18) circle (8pt);
    
    \def\s{9}
    \draw[very thick, purple] (0,0) rectangle +(\s,\s);
    \draw[very thick, blue] (2,4) rectangle +(\s,\s);
    \draw[very thick, cyan] (4,8) rectangle +(\s,\s);
    \draw[very thick, green] (6,18) -- (6,12) -- (6+\s,12) -- (6+\s,18)
        (6,0) -- (6,\s-6) -- (6+\s,\s-6) -- (6+\s,0);
    \draw[very thick, orange] (8,18) -- (8,16) -- (8+\s,16) -- (8+\s,18)
        (8,0) -- (8,+\s-2) -- (8+\s,\s-2) -- (8+\s,0);
    \draw[very thick,red] (18,2) -- (10,2) -- (10,2+\s) -- (18,2+\s)
        (0,2) -- (1,2) -- (1,2+\s) -- (0,2+\s);
    \draw[very thick,pink] (18,6) -- (12,6) -- (12,6+\s) -- (18,6+\s)
        (0,6) -- (3,6) -- (3,6+\s) -- (0,6+\s);
    \draw[very thick,magenta] (18,10) -- (14,10) -- (14,18)
        (14,0) -- (14,\s-8) -- (18,\s-8)
        (0,10) -- (5,10) -- (5,18)
        (0,\s-8) -- (5,\s-8) -- (5,0);
    \draw[very thick,brown] (18,14) -- (16,14) -- (16,18) 
        (16,0) -- (16,\s-4) -- (18,\s-4)
        (0,14) -- (7,14) -- (7,18) 
        (0,\s-4) -- (7,\s-4) -- (7,0);
    
    \node[below] at (0,0) {$0$};
    \node[left] at (0,0) {$0$};
    \node[below] at (18,0) {$2n$};
    \node[left] at (0,18) {$2n$};
    \node[yellow!60!black] at (9.5,19) {$\Gamma$};
    \node[purple,below] at (2,0) {$A_0$};
    \node[blue,below right] at (2,4) {$A_1$};
    \node[cyan,below right] at (4,8) {$A_2$};
    \node[green,below right] at (6,12) {$A_3$};
    \node[orange,below right] at (8,16) {$A_4$};
    \node[red,below right] at (10,2) {$A_5$};
    \node[pink,below right] at (12,6) {$A_6$};
    \node[magenta,below right] at (14,10) {$A_7$};
    \node[brown,below right] at (16,14) {$A_8$};
\end{tikzpicture}
\caption{Construction of a $(2,2)$-agreeable uniform toroidal society with $n=9$ voters.}
\label{torus_construction}
\end{center}

\end{figure}

Since any two $n \times n$ squares placed in $X$ with sides parallel to the coordinate axes must intersect, the society $S$, constructed above, is $(2,2)$-agreeable.

Recall that $\ell = \frac{n-1}{4}$. We now show that $a(S) = \ell+1$.
Since corresponding corners of each approval set lie on equally-spaced points on $\Gamma$, each approval set has the same pattern of intersection with the others.
Thus, without loss of generality, it suffices to restrict our attention to one approval set, $A_0$. First, the top-right corner point of $A_0$ is contained in $A_1, \ldots, A_\ell$, so $a(S) \ge \ell+1$.
We next show that no point of $A_0$ is contained in more than $\ell$ other approval sets, and hence $a(S) \le \ell+1$ as well.

Consider the four approval sets $A_1, A_{1+\ell}, A_{1+2\ell}, A_{1+3\ell}$, illustrated in \cref{set_intersection}.
Observe that each of these approval sets overlaps a different corner of $A_0$, and that no two of these approval sets intersect within $A_0$.
(We invite the reader to verify this by examining the coordinates of these sets within $X$.)

More generally, for $j \in \{1, \ldots, \ell\}$, the four approval sets $A_j, A_{j+\ell}, A_{j+2\ell}, A_{j+3\ell}$ each intersect one corner of $A_0$, but do not intersect each other within $A_0$. 
To see this, observe that the lower-left corner of $A_j$ is $(2j,4j)$, the upper-right corner is $A_{j+\ell}$ is $(2j+(n-1)/2, 4j-1)$, and so on.
That is, the approval sets $A_1, \ldots, A_{n-1}$ can be partitioned into $\ell$ collections $\C_j = \{ A_j, A_{j+\ell}, A_{j+2\ell}, A_{j+3\ell} \}$, each containing four approval sets, such that for each $j$, no two members of $\C_j$ intersect within $A_0$.

Thus, no point of $A_0$ is contained in more than one of the approval sets in any $C_j$, for any $j \in \{ 1, \ldots, \ell \}$, and so no point of $A_0$ is contained in more than $\ell$ other approval sets, so $a(S) \le \ell+1$.
Therefore, $a(s) = \ell+1 = \frac{n-1}{4}+1$.

\begin{figure}[h]
\begin{center}
\begin{tikzpicture}[scale=0.35]
	\draw[very thick, purple] (0,0) rectangle +(9,9);
    \draw[very thick, blue] (2,10) -- (2,4) -- (11,4);
    \draw[very thick, green] (6,-1) -- (6,3) -- (11,3);
    \draw[very thick, red] (-2,2) -- (1,2) -- (1,10);
    \draw[very thick, magenta] (-2,1) -- (5,1) -- (5,-1);
    \node[purple] at (3.5,2.5) {$A_0$};
    \node[blue,right] at (9,6.5) {$A_1$};
    \node[green,right] at (9,1.5) {$A_{1+\ell}$};
    \node[red,left] at (0,5.5) {$A_{1+2\ell}$};
    \node[magenta,left] at (0,0) {$A_{1+3\ell}$};
\end{tikzpicture}
\caption{The collection $\C_1 = \{ A_1, A_{1+\ell}, A_{1+2\ell}, A_{1+3\ell} \}$ contains four approval sets, no two of which intersect within $A_0$.}
\label{set_intersection}
\end{center}

\end{figure}
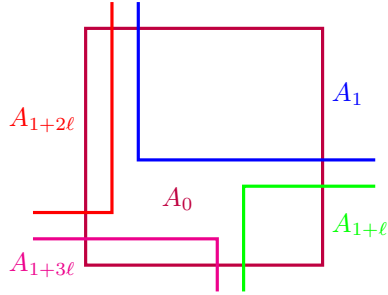

For $n \not\equiv 1 \pmod{4}$, let $\ell = \left\lceil \frac{n-1}{4} \right\rceil$. 
First construct the $(2,2)$-agreeable toroidal society for $4\ell + 1$ voters as above (the number of voters in this society is up to $3$ more than $n$), which has agreement number $a(S) = \ell + 1$.
Then remove the extra approval sets; this cannot increase the agreement number $a(S)$.
Since we do not remove $A_0, A_1, \ldots, A_\ell$, all of which intersect, the agreement number remains $\ell + 1$.
This produces a $(2,2)$-agreeable toroidal society $S$ with $a(S) = \ell + 1 = \left\lceil \frac{n-1}{4} \right\rceil + 1$.

\end{example}

The uniform toroidal societies constructed above allow us to establish the sharpness of \cref{main_cor,tor_bound}, as stated in the following corollary.

\begin{corollary} The lower bound on the agreement number of $(2, 2)$-agreeable toroidal societies given by \cref{main_cor} is sharp when $|S| \equiv 0$ or $1 \pmod{4}$, and the lower bound given by \cref{tor_bound} is sharp when $|S| \equiv 0$, $1$, or $3 \pmod{4}$.
\end{corollary}

\begin{proof}
By \cref{main_cor}, the agreement number of a $(2, 2)$-agreeable toroidal society satisfies the following inequality: $a(S) > \frac{|S|}{4}$.  When $|S| \equiv 0$ or $1 \pmod{4}$, this inequality is equivalent to
\[ a(S) \geq \left\lceil \frac{|S|-1}{4} \right\rceil + 1. \]
\cref{ex:uniformToroidal22} provides a family of $(2, 2)$-agreeable toroidal societies that achieve this lower bound.  Therefore, the bound in \cref{main_cor} is sharp.

The lower bound given by \cref{tor_bound} is equivalent to that given by \cref{main_cor} for $(2,2)$-agreeable toroidal societies when $|S| \equiv 0$ or $1 \pmod{4}$.
When $|S| \equiv 3 \pmod{4}$, the societies in \cref{ex:uniformToroidal22} achieve the lower bound given by \cref{tor_bound}. Therefore, the bound in \cref{tor_bound} is also sharp in these cases.
\end{proof}

Note that the agreement number of the uniform toroidal societies in \cref{ex:uniformToroidal22} is one greater than the bound given in \cref{main_cor,tor_bound} when $|S| \equiv 2 \pmod{4}$.
However, we suspect that the agreement number of the societies in \cref{ex:uniformToroidal22} is lowest possible for $(2,2)$-agreeable toroidal societies.
We offer the following conjecture.

\begin{conjecture}
The lower bound on the agreement number of $(2, 2)$-agreeable toroidal societies given by \cref{main_cor,tor_bound}, $ a(S) > \frac{|S|}{4}$, is not sharp for $|S| \equiv 2 \pmod{4}$.
\end{conjecture}

We invite the reader to investigate the sharpness of our lower bounds for the agreement number of $(k,m)$-agreeable toroidal societies where $k$ and $m$ are not both $2$.

\section*{Acknowledgments}
The authors started this project at the 2014 AMS Mathematics Research Community on Algebraic and Geometric Methods in Applied Discrete Mathematics, which was supported
by NSF DMS-1321794.
We express great thanks to Francis Su for inviting us to consider approval voting problems through a focus group at the MRC, which provided inspiration for this paper, and also for his encouragement and helpful comments throughout the course of this work.
We thank Kathryn Nyman for her leadership and discussions at the MRC.
We also thank the two anonymous referees for their valuable suggestions.

\end{document}